\documentclass{amsart}
\usepackage{graphicx}
\usepackage{verbatim}

\usepackage{amssymb}
\usepackage{tikz}
\usetikzlibrary{decorations.markings}
\usepackage{enumerate}
\usepackage{mathtools}
\usepackage[T1]{fontenc}
\usepackage[backend=bibtex]{biblatex}
\bibliography{sninvariant}

\title{A family of link concordance invariants from perturbed $sl(n)$ homology.}
\begin{document}
\author{Gahye Jeong}
\vspace{-10mm}
\maketitle
\newtheorem{thm}{Theorem}
\newtheorem{df}{Definition}
\newtheorem{eg}{Example}
\newtheorem{prop}{Proposition}
\newtheorem{cor}{Corollary}
\newtheorem{rmk}{Remark}
\newtheorem{lem}{Lemma}

\newcommand{\oo}{\mathbf{o}}
\newcommand{\qgr}{\text{qgr}}

\numberwithin{equation}{section}
\numberwithin{figure}{section}
\begin{abstract}
 We define a family of link concordance invariants $\left\{ s_n \right\}_{n=2,3, \cdots}$. These link concordance invariants give lower bounds on the slice genus of a link $L$. We compute the slice genus of positive links. Moreover, these invariants give lower bounds on the link splitting number of a link. Especially, this new lower bound determines the splitting number of positive torus links. This is a generalization of Lobb's knot concordance invariants $\left\{ s_n \right\}$, obtained from Gornik's spectral sequence.
\end{abstract}
\vspace{10mm}
\section{Introduction}

 During the past two decades, a variety of knot homologies have been introduced as categorifications of knot polynomials. Knot homologies may be related by spectral sequences. In particular, some spectral sequences start from a specific knot homology and end in a trivial homology. The gradings or the filtration levels appeared in the last page of these spectral sequences often define knot invariants. Various knot concordance invariants have been constructed in this way.

 In \cite{khovanov2002}, Khovanov constructed a chain complex with a homological grading and a quantum grading for a link diagram, whose homology group is a link invariant. The Khovanov homology categorifies the Jones polynomial. By perturbing the differential of the Khovanov chain complex, Lee \cite{lee} defined a spectral sequence from the Khovanov homology to $\mathbb{C}^2$ for knots. From the quantum grading of the last page of Lee's spectral sequence, Rasmussen \cite{rasmussen} extracted a homomorphism $s$ from the knot concordance group to $\mathbb{Z}$. Rasmussen's $s$ invariant gives a slice genus bound for a knot. This lower bound provides a simpler proof of Milnor's conjecture on the slice genus of torus knots than the pre-existing proofs. 

 In \cite{khovanov}, Khovanov and Rozansky constructed $sl(n)$ homology for an oriented link in $S^3$ by using a matrix factorization, which is a categorification of $sl(n)$ polynomial. In \cite{gornik}, Gornik defined a spectral sequence from the Khovanov-Rozansky $sl(n)$ knot homology to $\mathbb{C}^n.$ Lobb and Wu independently proved a family of inequalities, which provide the lower bounds on the slice genus of knots from this spectral sequence in \cite{lobb1, wu}. In \cite{lobb2}, Lobb constructed a family of knot concordance invariants $\left\{ s_n \right\}$, which give slice genus bounds. Moreover, Lobb showed that $s_n$ determines Gornik's homology, that is the last page of this spectral sequence, completely for knots. Lobb conjectured that $\left\{ s_n \right\}$ might have a linear dependence. However, Lewark showed that some subsets of $\left\{s_n : n=2,3, \cdots \right\}$ are linearly independent in \cite{lewark}. Then it is quite plausible that $\left\{s_n \right\}$ may provide more information about knots. In this paper, we want to generalize Lobb's $s_n$ invariants for links. We introduce a family of invariants $s_n (L) \in \mathbb{Z}$ for an oriented link $L \subset S^3$ and $n=2,3, \cdots$, which is equal to Lobb's $s_n$ invariant when $L$ is a knot. Then the link invariants $\left\{s_n\right\}$ share similar properties with the original $s_n$ knot invariants.

 These $s_n$ link invariants give obstructions to sliceness of links. There are two ways to generalize the concept of sliceness for links. Following the definitions in \cite{fox}, we say $L$ is {\it slice in the strong sense} if there is a proper and smooth embedding of $l$ disks $D_1, \cdots D_l$ into $B^4$ such that $\partial D_i$ is $i$ th component of $L$ for $i=1,2, \cdots, l$. On the other hand, we say $L \subset S^3$ is {\it slice in the ordinary sense} if there is an oriented connected genus $0$ surface $F$ which is properly and smoothly embedded in $B^4$ such that $\partial F = L$. The slice genus $g_4 (L)$ is defined to be the minimal genus of an oriented connected surface $F$ which is properly and smoothly embedded in $B^4$ with the boundary $L$. Then $g_4 (L) = 0 \iff L$ is slice in the ordinary sense. Theorem \ref{main} shows how $s_n$ link invariants are related to the ordinary sense of sliceness. 
\begin{thm}\label{main} Suppose $F$ is a properly embedded oriented surface in $B^4$ with genus $g(F)$ whose boundary is equal to $L \subset S^3$. Let $k$ be the number of connected components of $F$. Then $$(n-1) (2k-1-\chi(F)) \ge s_n(L) \ge (n-1) (\chi (F) - 1).$$ Note that $\chi(F) = 2k-2g(F) - l$. 
 In particular, when $F$ is connected, $$|\frac{s_n (L)}{1-n}| \le 2g (F) + l-1.$$ Therefore, $$|\frac{s_n (L)}{1-n}| \le 2g_4 (L)+ l-1.$$
\end{thm}
When $F$ is a collection of disjoint disks, $k =\chi(F) = l$. Then the following corollary shows that the strong sliceness determines $s_n$ link invariants completely.
\begin{cor}\label{str} If an $l$ component link $L$ is strongly slice, then $s_n (L) = (n-1)(l-1)$. 
\end{cor}

  More generally, the link invariants $\left\{s_n \right\}$ are concordance invariants. 
\begin{thm}\label{thm2}
 When two $l$ component links $L_0, L_1 \subset S^3$ are concordant, that is, there exists an embedding $f: (S^1 \sqcup \cdots \sqcup S^1) \times [0,1] \longrightarrow S^3 \times [0,1]$ such that $f|_{(S^1 \sqcup \cdots \sqcup S^1)  \times  \left\{0\right\}} = L_0$ and $f|_{(S^1 \sqcup \cdots \sqcup S^1) \times \left\{1\right\}} = L_1 \subset S^3 \times \left\{1\right\}$, then $$s_n (L_0) = s_n(L_1).$$
\end{thm}

\begin{thm} \label{thm3}Let $l$ be the number of components of a link $L$. $\bar{L}$ denotes a mirror of a link $L$. Let $L_1, L_2$ also be links. $L_1 \#_{i_1 = i_2} L_2$ is obtained from $L_1$ and $L_2$ by connecting  the $i_1$th component of $L_1$ and the $i_2$th component of $L_2$. Then,
\begin{align}
\label{union}s_n (L_1 \sqcup L_2) = s_n (L_1) + s_n(L_2) + n-1 \\
\label{connected}s_n(L_1) + s_n(L_2)  = s_n (L_1 \#_{i_1 = i_2} L_2)  \\
\label{mirror}  0 \le s_n (L) + s_n(\bar{L}) \le (2l-2)(n-1).
\end{align}
\end{thm}

 With these properties, we study applications of $\left\{s_n \right\}$ on other link invariants. We compute the slice genus for positive links by using Theorem \ref{main}. 

\begin{thm}\label{positive}
 For a positive link $L$, $s_n(L) = (1-n)(c-r+1)$, where $r$ the number of components of the oriented resolution and $c$ the number of crossings for a link diagram of $L$. Moreover, Theorem \ref{main} implies $g_4 (L) = g_3 (L) = \frac{2-(r-c+l)}{2}$, where $g_3 (L)$ denotes a minimal genus of the Seifert surface of $L$.
\end{thm}

Then the corollary follows naturally.
\begin{cor} For torus links $T(p,q)$, $p > 0, q > 0$, 
\begin{center}
$\displaystyle g_4 (T(p,q)) = \frac{(p-1)(q-1)+1-\text{gcd} (p,q)}{2}$,
\end{center}
where $gcd(p,q)$ denotes the greatest common divisor of $p$ and $q.$ 
\end{cor}

 From \cite{batsonseed}, the splitting number of a link $sp(L)$ is defined by the minimal number of the crossing changes between different components, to make the link split. We want to show that $s_n (L)$ also gives a bound on the link splitting number $sp(L)$.
\begin{thm}\label{change} Let $L_{+}$ be a link in $S^3$ with a marked positive crossing. $L_{-}$ is the link obtained from $L_{+}$ by changing the marked positive crossing to a negative crossing. Then, 
 $$|s_n (L_{+}) - s_n (L_{-})| \le 2(n-1).$$
\end{thm}

\begin{thm}\label{splitting} Let $L  = L_1 \cup \cdots \cup L_l$ be an $l$ component link in $S^3$. Each $L_i$ is a knot for $i=1,2, \cdots, l$. Then $$|s_n (L) - \sum_{i=1}^l s_n (L_i)- (n-1)(l-1)| \le 2 (n-1) sp(L).$$.
\end{thm}

 Moreover, we will compute the splitting number of positive torus links using Theorem \ref{splitting}.
\begin{cor}  Let $p,q>0$ be coprime. Let $L$ be a positive torus link $T(lp, lq)$ for a positive integer $l$. Then $$sp(L) = \frac{l(l-1)}{2} pq .$$
\end{cor}

 When $\mathcal{L}$ is a link and $o$ denotes the orientation on the link $\mathcal{L}$, Beliakova and Wehrli defined an oriented link invariant $s(\mathcal{L},o)$ in \cite[Section 6]{beliakova} from Lee's homology. For an oriented link $L \subset S^3$, we denote this link invariant $s(L)$. We show that Beliakova and Wehrli's link invariant $s$ is equivalent to $-s_2$.

\begin{thm}\label{equal}
For an oriented link L, $s(L)$ is equivalent to $-s_2 (L)$, which will be defined in Definition \ref{defsn}.
\end{thm}
 In \cite{murasugi, tristam}, the authors also give a lower bound for $2g_4(L) + l -1$ by using the Tristam-Levine signatures. However, Beliakova and Wehrli provide examples that $s(L)$ gives the stronger obstruction to the regular sliceness than Tristam-Levine signature. We may attain better obstructions on $2g_4(L) + l -1$ from $s_n$ link invariants since $\left\{ s_n \right\}$ contain more information than a single $s_2$. 

  Recently, Lobb and Lewark defined a family of knot concordance invariants which arose from both unreduced and reduced Khovanov-Rozansky cohomologies with separable potentials in \cite{lobblewark}. We expect that the results in \cite{lobblewark} also might be generalized to links. Moreover, special cases of the deformation theorem proved in \cite{rosewedrich} proposed spectral sequences starting from one perturbed homology and abutting to the other perturbed homology. By tracking the quantum filtrations of the above spectral sequences, we might obtain the relations between the invariants defined in \cite{lobb2} and \cite{lobblewark}. The author may return to these two projects in a future paper.

 The paper is organized as follows. In Section 2, we review the description of the perturbed $sl(n)$ homology. Next, in Section 3, we review how to construct a map between perturbed $sl(n)$ link homologies when a link cobordism between two links is given. In Section 4, we define $s_n$ invariants of links. In Section 5, we prove the theorems and properties of new link invariants, given in the introduction.

\section*{Acknowledgement}
 I would like to thank Professor Yi Ni for introducing this topic to me and for giving advices to me. I would also like to thank Professor Andrew Lobb, for suggesting the possible further direction. The author is partially supported from Samsung Scholarship.
\section{Perturbed $sl(n)$ homology}
 In this section, we will review a description of the perturbed $sl(n)$ homology. We follow the description in \cite{gornik, lobb1}. We consider an $l$ component link $L$ and fix a link diagram $D$ of $L$ with $N$ crossings. Then for every $v \in \left\{ 0,1 \right\} ^N$, we do a $v_i$ resolution for the $i$th crossing as in Figure \ref{12}. $D_v$ denotes the corresponding resolution. Then we call $D_v$ a link resolution. Let $\oo = (0,0 \cdots ,0) \in \left\{0,1\right\}^N$. Henceforth, we call $D_{\oo}$ the oriented resolution and $w(D)$ the writhe of $D$.

 We fix a link resolution $\Gamma$ and a degree $n+1$ polynomial $\omega (x) \in \mathbb{C} [x]$. Then we place marks on thin edges of $\Gamma$. There is a formal variable per each mark. We assign a 2-periodic complex  $C_{\omega} (\Gamma)$ by a tensor product of matrix factorizations over the polynomial ring $\mathbb{C} [ x_1, \cdots, x_k]$ for formal variables $x_1, \cdots, x_k$. We have $\mathbb{Z}/2\mathbb{Z}$-grading on $C_{\omega} (\Gamma)$ and the differential maps $d^{\omega}_0, d^{\omega}_1$ which satisfy $d^{\omega}_0 d^{\omega}_1 = d^{\omega}_1 d^{\omega}_0= 0$. It can be expressed by $$C^0_{\omega} (\Gamma) \xrightarrow{d^{\omega}_0} C^1_{\omega} (\Gamma) \xrightarrow{d^{\omega}_1} C^0_{\omega} (\Gamma).$$ See \cite{lobb1, khovanov} for the full descriptions.

\begin{rmk}\label{dis} For two disjoint link resolutions $\Gamma_1, \Gamma_2$, it is known that $$C_{\omega} (\Gamma_1 \sqcup \Gamma_2) = C_{\omega} (\Gamma_1) \otimes_{\mathbb{C}} C_{\omega} (\Gamma_2)$$ from \cite[Proposition 28]{khovanov}.
\end{rmk}
 Let $d^{\omega}$ be $d^{\omega}_0 + d^{\omega}_1$. With this total differential $d^{\omega}$, we define a homology group $H_{\omega} (\Gamma)$ of $C_{\omega} (\Gamma)$. Then, we define a quantum $\mathbb{Z}$-grading $q$ on the complex $C_{\omega} (\Gamma)$ by $q(x) = 2$ for each formal variable $x$. We say a map $f : C_{\omega} (\Gamma) \longrightarrow C_{\omega} (\Gamma)$ is {\it homogeneous with degree $k$} if $q(f(x)) = q(x) +k$.  The boundary map $d^{\omega}_0, d^{\omega}_1$ is a sum of homogeneous maps; each homogeneous summand with degree $-n-1+2k$ corresponds to the degree $k$ term of $\omega(x)$ for $k= 0,1, \cdots, n+1.$

 For $\omega(x) = x^{n+1}$, the boundary maps are homogeneous with degree $(n+1)$. Therefore, the quantum $\mathbb{Z}$-grading $q$ on the chain complex induces the quantum $\mathbb{Z}$-grading on the homology group. However, for $\omega(x)= x^{n+1} - (n+1) x$, the boundary maps are a sum of two parts: one with degree $(n+1)$ and the other with degree $(-n+1)$. Then $q$-grading on the chain complex cannot induce the quantum grading on the homology group, but can induce a $\mathbb{Z}$-filtration on the homology group \cite[Section 2]{gornik}. Recall that a filtration on $H_{\omega} (\Gamma)$ indexed by $\mathbb{Z}$ is a collection of subsets $\left\{ \mathcal{F}^i H_{\omega} (\Gamma)\right\}_{i \in \mathbb{Z}}$ such that $$\cdots \subset \mathcal{F}^j H_{\omega} ( \Gamma) \subset \mathcal{F}^{j+1} H_{\omega} (\Gamma) \subset \cdots$$ for all $j \in \mathbb{Z}$. We call this induced filtration the quantum filtration. Note that the boundary maps for $\omega = x^{n+1} - (n+1)x$ are not homogeneous with $\mathbb{Z}$-grading, but homogeneous with $\mathbb{Z}/2n\mathbb{Z}$-grading, which is naturally determined from $\mathbb{Z}$-grading. Therefore, $\mathbb{Z}/2n\mathbb{Z}$-grading can be also defined on the homology level. We will return to this $\mathbb{Z}/2n\mathbb{Z}$-grading in Section 4.
\begin{rmk}\label{rmk1}
 From the definition of ${H}_{\omega}$, it is known that $${H}_{\omega} (\text{$r$ component unlink}) = \mathbb{C} [ x_1, \cdots, x_r ] / (\omega' (x_1) , \omega' (x_2), \cdots \omega' (x_r) ) \left\{(1-n)r\right\},$$ where $\left\{-\right\}$ means a shift in the quantum filtration.
\end{rmk}

\begin{figure}
\begin{tikzpicture}
	\draw[->] (-2,-0.5)--(-1,0.5);
	\draw (-1,-0.5)--(-1.4,-0.1);
	\draw[->] (-1.6,0.1)--(-2,0.5);

	\draw[->] (2,-0.5)--(1,0.5);
	\draw (1,-0.5)--(1.4,-0.1);
	\draw[->] (1.6,0.1)--(2,0.5);

	\node at (-1.5,-1.2) {Positive crossing};
	\node at (1.5,-1.2) {Negative crossing};
\end{tikzpicture}

\vspace{4mm}
\begin{tikzpicture}
	\node at (0,0.5) {$\longrightarrow$};
	\node at (0,-0.5) {$\longleftarrow$};
	\node at (0,1) {$\chi_0$};
	\node at (0,-1) {$\chi_1$};
	\node at (-2,-2) {$\Gamma_0$};
	\node at (2,-2) {$\Gamma_1$};
	\node at (-2,1.5) {$0$ resolution};
	\node at (2,1.8) {$1$ resolution};

	\draw[->] (-3,-1) arc (-40:40:1.5);
	\draw[->] (-1,-1) arc (220:140:1.5);
	\draw[->] (2,0.5)--(1.2,1.5);
	\draw[->] (2,0.5)--(2.8,1.5);
	\filldraw (1.9,0.5)--(2.1,0.5)--(2.1,-0.5)--(1.9,-0.5);
	\draw (2,-0.5)--(1.2,-1.5);
	\draw (2,-0.5)--(2.8,-1.5);

\end{tikzpicture}
	\caption{}\label{12}
\end{figure}
 Let $\Gamma_0, \Gamma_1$ be two resolutions of the link, which locally differ as in Figure \ref{12}. We have two chain maps between the matrix factorizations
$$\chi_0 : C_{\omega} (\Gamma_0) \longrightarrow C_{\omega} (\Gamma_1)$$
$$\chi_1 : C_{\omega} (\Gamma_1) \longrightarrow C_{\omega} (\Gamma_0).$$

   Now we define an edge set on $\left\{0,1 \right\}^N$. Suppose $u,v \in \left\{0,1 \right\}^N$ satisfy $v_i - u_i = 1, u_j = v_j$ for all $j \ne i$. When the $i$ th crossing is positive, we add an edge from $u$ to $v$. Otherwise, we add an edge from $v$ to $u$. Moreover, for each edge $e : u \longrightarrow v$, $\Phi_{uv} : C_{\omega} ( D_u) \longrightarrow C_{\omega} (D_{v})$ is described by $\chi_0$ or $\chi_1$ depending on the local diagrams of $D_u$ and $D_v$. Now we define the grading $s(v)$ to be the number of $1 \le i \le N$ satisfying one of the following: 
\begin{enumerate} \item $v_i =1$ and the $i$ th crossing is positive, \item $v_i=0$  and the $i$ th crossing is negative. \end{enumerate} For every element  $x \in C_{\omega} (D_v)$, we assign the homological grading $h (x) = s(v) - \frac{1}{2} ( N + w(D)) $. Then $\Phi_{uv}$ increases the homological grading by 1.

 The chain complex is given by $$CKh^k_{\omega} (D)= \bigoplus_{s(v) = k + \frac{1}{2} (N+w(D))} H_{\omega} (D_v) \left\{ (1-n)w(D) - k \right\},$$ where $w(D)$ is the writhe of the diagram and $k$ denotes the homological grading. Note that the term $\frac{1}{2} ( n + w(D))$ is added to make the homological grading of ${H}_{\omega} (D_{\oo})$. 

 The differential is defined to be
$$\partial_{\omega} = \sum_{\substack{e:u\longrightarrow v \\u,v \in \left\{0,1\right\}^n }} \Phi_{uv}.$$ This differential map preserves the quantum filtration. The homology group of $(CKh_{\omega} (D) , \partial_{\omega})$ is denoted by $HKh_{\omega} (D)$. 

From Remark \ref{dis}, \begin{equation}\label{dis2}HKh_{\omega} (D_1 \sqcup D_2) = HKh_{\omega} (D_1) \otimes_{\mathbb{C}} HKh_{\omega} (D_2).\end{equation} for two disjoint link diagrams $D_1, D_2$.

 In \cite{wu}, it is shown that this chain equivalence class does not depend on the choice of the diagram, hence it is a link invariant. Thus, we can define the perturbed $sl(n)$ link homology $HKh_{\omega} (L)$ by $HKh_{\omega} (D)$. 

 For $\omega = x^{n+1}$, $HKh_\omega$ is the original Khovanov-Rozansky $sl(n)$ homology, constructed in \cite{khovanov}. Since $d^{\omega}_0, d^{\omega}_1$ are homogeneous, the homology group has a well-defined quantum grading. For $\omega = x^{n+1} - (n+1)x$, $HKh_{\omega}$ is equivalent to Gornik's homology, which is originally defined in \cite{gornik}. Henceforth, we use the notation $CKh'_n$ and $HKh'_n$ for $CKh_{x^{n+1}-(n+1)x}$ and $HKh_{x^{n+1}-(n+1)x}$ respectively. 

 We want to describe the canonical generators of $HKh'_n (D)$. Each generator has one-to-one correspondence with a map $\psi : \left\{\text{components of $\mathcal{L}$} \right\} \longrightarrow \left\{1, \xi, \cdots, \xi^{n-1} \right\}$ for $\xi = \frac{2 \pi i}{n}$ in the following way. 

\begin{thm}\cite[Theorem 2]{gornik}
 Let $L$ be a link and $D$ be a link diagram of $L$. For each crossing in $D$, there are two strands involved. If the two strands have the same $\psi$ values, we do a 0-resolution for the crossing. Otherwise, we do a 1-resolution for the crossing. $D_v$ denotes the corresponding link resolution. The corresponding generator lives in ${H}(D_v)$. Then the element which corresponds to $\psi$ is defined to be 
$$ g^{D}_{\psi}  := \displaystyle \prod_{e :  \text{edge}} \frac{{X_e}^n - 1}{ X_e - \psi(\text{component including $e$})}.
$$
  Now we call elements in $\left\{g^D_{\psi}| \psi : \left\{\text{components of $\mathcal{L}$} \right\} \longrightarrow \left\{1, \xi, \cdots, \xi^{n-1} \right\}\right\}$ canonical generators. 
\end{thm}

\begin{df}
For a link diagram $D$, we define $g^D_i := g^D_{\psi}$, where $\psi$ is a constant map whose image is $\left\{ \xi^i \right\}$  for every $i= 0, 1, \cdots, n-1$. Note that all $g^D_i$ live in ${H} (D_\mathbf{o})$ for $\mathbf{o} = (0, 0 \cdots ,0) \in \left\{0,1\right\}^N$. 
\end{df}

 Henceforth, we fix $n$. $CKh'$ and $HKh'$ denote $CKh'_n$ and $HKh'_n$ respectively. Also $H$ denotes $H_{x^{n+1} - (n+1) x}$.

\section{Link cobordism.} 
 In this section, we want to construct a chain map $CKh'(\Sigma)$ for a link cobordism $\Sigma : L_0 \longrightarrow L_1$ in  $S^3 \times [0,1]$. If a surface $\Sigma \subset S^3 \times [0,1]$ is an oriented surface with boundaries such that $\partial \Sigma = L_0 \sqcup L_1$ and $\partial \Sigma \cap (S^3 \times \left\{i \right\} )= L_i$, then we call $\Sigma$ a link cobordism from $L_0$ to $L_1$.   It is known that $\Sigma$ can be expressed as a finite sequence of elementary moves, that is, Morse moves and Reidemeister moves. Therefore, it is enough to construct the corresponding chain maps for Morse moves and Reidemeister moves. 

\subsection{Morse moves.}
 There are three kinds of Morse moves which locally change the link diagram: 0-handle move, 1-handle move and 2-handle move. Figure \ref{morsemoves} describes how Morse moves change the link diagram. First, the 0-handle move is the operation that creates a new unknotted component disjoint from all other link components. There is a link cobordism $S_0 :D \longrightarrow D \sqcup U$ corresponding to the $0$-handle move when $D$ is a link diagram and $U$ is an unknot. $S_0$ is given by $D \times [0,1] \sqcup D^2$, where the boundary of $D^2$ is $U$. This cobordism is depicted in Figure \ref{02handle}.  
\begin{figure}
	\includegraphics[scale=0.4]{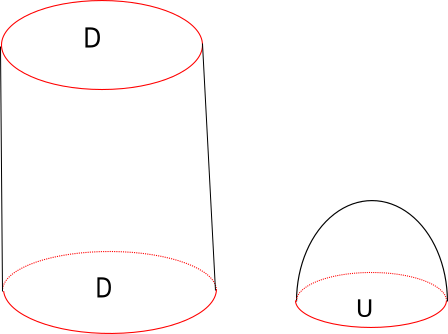}\hspace{15mm}\includegraphics[scale=0.4]{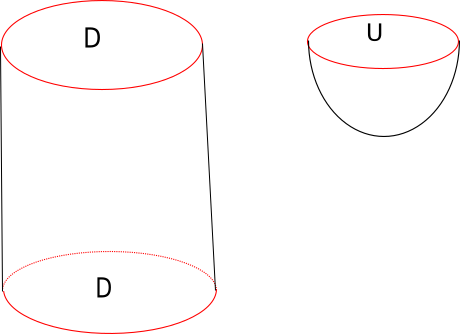}\caption{$S_0$ and $S_2$}\label{02handle}

\end{figure}
Second, the 2-handle move is a reverse operation of the 0-handle move. The 2-handle move removes an unknotted and unlinked component by adding a disjoint disk whose boundary is supported in $S^3 \times \left\{0\right\}$. The corresponding cobordism $S_2 : D \sqcup U \longrightarrow D$ is described in Figure \ref{02handle}. 
\begin{figure}
\end{figure}
Lastly, the 1-handle move is described by a saddle addition between two arcs of the link diagram. Figure \ref{1handle} describes the local picture of a cobordism $S_1$ corresponding to the $1$-handle move.
\begin{figure}
	\includegraphics[scale=0.6]{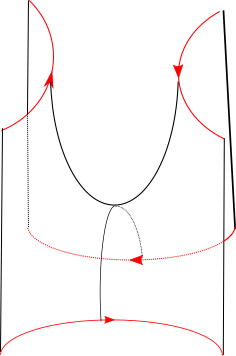}
	\caption{$S_1$}\label{1handle}
\end{figure}

\begin{figure}
\begin{tikzpicture}[scale=0.8]
	\draw[decoration={markings, mark=at position 0.125 with {\arrow{>}}},
        postaction={decorate}] (3,1.5) circle (1cm);
	\draw[<-] (2, -0.5) arc (240:300:2);
	\draw[<-] (4, -2) arc (60:120:2);
	\node at (0,1.85) {2-handle};
	\node at (0, 1.15) {0-handle};
	\node at (0, -1.2) {1-handle};
	\draw[->] (1.5,1.6)--(-1.5,1.6);
	\draw[<-] (1.5,1.4)--(-1.5,1.4);
	\draw[->] (-2, -0.5) arc (150:210:2);
	\draw[<-] (-4, -0.5) arc (30:-30:2);
	\draw[<-] (1.5,-1.5)--(-1.5,-1.5);
\end{tikzpicture}
	\label{morsemoves}\caption{}
\end{figure}

  In \cite[Section 3]{lobb1}, Lobb showed that when two link diagrams $D_0, D_1$ differ by a Morse move, there is a chain map from $CKh' (D_0)$ to $CKh'(D_1)$. Before describing the chain maps explicitly, we fix the following notations. From Remark \ref{rmk1}, $HKh' (U) = \mathbb{C}[x] / (x^n -1) \left\{1-n\right\}$. The canonical generator of $HKh' (U)$ is equal to $g_i := \frac{x^n -1}{x- \xi^i}$ when $\xi = \frac{2 \pi i}{n}$. Let  $\Sigma_n$ be a set of all roots of $x^n -1$. Then $\Sigma_n = \left\{1, \xi, \xi^2, \cdots \xi^{n-1} \right\}$. Note that $\displaystyle \sum_{j=0}^{n-1} (\prod_{\xi^j \in \Sigma_n \setminus \xi^i} \frac{1}{\xi^j - \xi^i}) g_i=1$.

 For a 0-handle map, we associate the map between chain complexes 
\begin{center}
$1 \otimes i : CKh' (D) \longrightarrow CKh' (D \sqcup U) = CKh' (D) \otimes CKh'(U)$
\end{center} of filtered degree $1-n$.
 $1 \otimes i$ sends a canonical generator $g \in CKh'(D)$ to $g \otimes 1 = \displaystyle \sum_{i=0}^{n-1} g \otimes (\prod_{\xi^j \in \Sigma_n \setminus \xi^i} \frac{1}{\xi^j - \xi^i}) g_i.$ 

 For a 2-handle map, we associate the map between chain complexes 
\begin{center}
$1 \otimes \epsilon : CKh' (D \sqcup U) = CKh' (D) \otimes CKh'(U) \longrightarrow CKh' (D)$
\end{center} of filtered degree $1-n$.
 $1 \otimes \epsilon$ sends a canonical generator $\displaystyle g \otimes g_i$ to $g$. 

 There are two kinds of 1-handle maps: a fusion map which decreases the number of components and a fission map which increases the number of components. First, let $S$ be a fusion cobordism from $D = L_1 \cup \cdots L_n$ to $D' = L'_1 \cup \cdots L'_{n-1}$, where $L'_i, L_i$ are knot diagrams. Suppose $S$ merges $L_{n-1}$ and $L_n$ into $L'_{n-1}$ and sends $L_i$ to $L'_i$ for $1 \le i \le n-2$. Then for each $\psi : \left\{L_1, \cdots, L_n\right\} \longrightarrow \Sigma_n$ satisfying $\psi(L_{n-1}) = \psi(L_n)$, we define $\bar{\psi} : \left\{ L'_1, \cdots, L'_{n-1} \right\} \longrightarrow \Sigma_n$ in the following way: $\bar{\psi} (L'_{n-1}) := \psi(L_{n-1}) = \psi(L_n)$ and $\bar{\psi} (L'_i)= \psi(L_i)$ for $1\le i \le n-2$. Lobb defined the map $CKh'(S)$ between the chain complex which is defined by the matrix factorization. $CKh'(S)$ induces a map $HKh'(S)$ on the homology level. If $\psi(L_{n-1}) = \psi(L_n)$, then $HKh'(S)$ sends $[g^D_{\psi}]$ to a nonzero multiple of $[g^{D'}_{\bar{\psi}}]$.  Otherwise it sends $[g^D_{\psi}]$ to 0. Moreover,  $HKh'(S)$ has filtered degree $n-1$.

 Next, let $S'$ be the fission cobordism from $D = L_1 \cup \cdots L_n$ to $D' = L'_1 \cup \cdots L'_{n+1}$. Suppose $S'$ splits the component $L_n$ into $L'_n$ and $L'_{n+1}$ and sends the component $L_i$ to $L'_i$ for $1 \le i \le n-1$. For each $\psi : \left\{L_1, \cdots, L_n\right\} \longrightarrow \Sigma_n$, $\bar{\psi} : \left\{L_1', \cdots, L'_{n+1} \right\}$ is defined in the following way: both $\bar{\psi} (L'_n)$ and $\bar{\psi} (L'_{n+1})$ are equal to $\psi(L_n)$ and $\bar{\psi} (L_i') = \psi (L_i)$ for $1 \le i \le n-1$. Then the corresponding cobordism map $HKh'(S')$ sends $[g^D_{\psi}]$ to a nonzero multiple of $[g^{D'}_{\bar{\psi}}]$. This map $HKh'(S')$ has filtered degree $n-1$. 

\subsection{Reidemeister moves.} Let $D_0, D_1$ be two link diagrams which differ by Reidemeister moves. In \cite[Proposition 4.2, 4.6, 4.8 and 4.9]{lobb2}, Lobb showed that there is always a chain map $\Phi$ from $CKh' (D_0)$ to $CKh'(D_1)$ which respects the homological grading and the quantum filtration when they differ by Reidemeister move I, II. Furthermore, $\Phi$ induces a map on the homology level which sends a canonical generator $[g^{D_0}_{\psi}]$ to a nonzero multiple of the corresponding canonical generator $[g^{D_1}_{\psi}]$.

  Also Wu independently showed that there is a chain map between $CKh'(D_0)$ and $CKh'(D_1)$ which preserves the homological grading and the quantum filtration, where $D_0, D_1$ differ by Reidemeister moves, in \cite[Proposition 5.9.]{wu}. This chain map induces an isomorphism on the homology level. Moreover, this map sends $[g^{D_0}_{\psi}]$ to a nonzero multiple of $[g^{D_1}_{\psi}]$.

\subsection{Link cobordism decompositions.}
 Given $\Sigma : L \longrightarrow L'$ in $S^3 \times [0,1]$, we construct $\Sigma$ as a product of elementary cobordisms. In other words, there is a sequence of link diagrams $D_0, D_1, \cdots, D_n$ such that $D_0$ is a link diagram for $L$, $D_n$ is a link diagram for $L'$ and two consecutive diagrams differ by a Morse move or a Reidemeister move. There are corresponding maps $CKh'(\Sigma_i)$ from $CKh' (D_{i-1})$ to  $CKh'(D_i)$ for $i=1,2, \cdots , n$. Then, we define $CKh'(\Sigma) = CKh'(\Sigma_n) \circ \cdots \circ CKh'(\Sigma_1).$ $CKh'(\Sigma)$ is a map of filtered degree $(1-n) \chi (\Sigma)$. In \cite{lobb1}, Lobb ordered elementary moves in the cobordism in the following way. 
 \begin{thm}\label{1.6}\cite[Theorem 1.6]{lobb1} Suppose $\Sigma_g : L_0 \longrightarrow L_1$ is a connected genus $g$ link cobordism between the two links $L_0$ and $L_1$. Suppose $\Sigma$ is embedded in $\mathbb{R}^4$. Let $D_i$ be the link projection of $L_i$. By removing $k$ disks from $\Sigma_g$, we can get a link cobordism $\bar{\Sigma} : L_0 \longrightarrow \bar{L_1}$, where $\bar{L_1}$ is a disjoint union of $L_1$ and $k$ component unlink. For some $k$, there exists a presentation of $\bar{\Sigma}$ as a sequence of elementary cobordisms with the following order:

\begin{enumerate}

\item The presentation begins with the diagram $D_0$

\item Then all the 0-handles of the presentation.

\item Then a sequence of Reidemeister I and II moves.

\item Then a sequence of fusion 1-handles, ending in a 1-component knot diagram.

\item Then $g$ fission 1-handles.

\item Then $g$ fusion 1-handles.

\item Then a sequence of Reidemeister I and II moves and fission 1-handles, ending in a diagram $\bar{D_1}$ of $\bar{L_1}$, which is the disjoint union of $D_1$ and 0-crossing diagrams of the unknot and diagrams as in Figure \ref{36}

\end{enumerate}

\end{thm}

\begin{figure}

\begin{tikzpicture}

\draw (0,0) circle (1cm);

\draw (0.7,0.9) arc (120:-120:1);

\draw (0.55,0.7) arc (135:225:1);

\draw (-0.3,1.1) arc (170:10:1);

\draw (-0.3, 0.8) arc (190:240:1);

\draw (1.7, 0.8) arc (-20:-70:1);

\draw (0.4,0.1) arc (255:285:1);

\end{tikzpicture}

\caption{A diagram of three component unlink.}\label{36}

\end{figure}
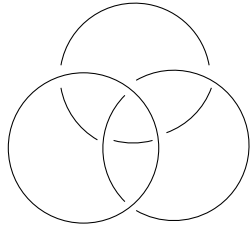
 We add the last step from $\bar{D_1}$ to $D_1$, which is a sequence of Reidemeister moves and the $2$ handles. Then we can fully order the elementary moves in the original connected cobordism $\Sigma$. If $S:D' \longrightarrow D''$ is an elementary cobordism except $0$-handle, then $HKh’(S) ([g^{D'}_i])$ is equal to a nonzero multiple of $[g^{D''}_i]$. Therefore, Theorem \ref{1.6} is quite useful when we track how the generators $[g^{D_0}_i], (i=0,1,2 \cdots, n-1)$ are mapped, because it puts all the $0$-handles in the first step. Moreover, even the cobordism is disconnected, we may achieve such ordering. We will visit this ordering of the elementary cobordisms again in Section 5 to prove theorems.
\section{Definition of $s_n$.}
 Let $D$ be a link diagram of a link $L \subset S^3$. In Section 2, we define the chain complex $CKh'(D)$, equipped with the quantum $\mathbb{Z}$-grading $q$ and the homolological $\mathbb{Z}$-grading $h$. The quantum $\mathbb{Z}$-grading $q$ induces a quantum $\mathbb{Z}$-filtration on the homology group $HKh'(D).$ We use the notation $$\cdots \mathcal{F}^j HKh'(D) \subset \mathcal{F}^{j+1} HKh'(D) \cdots$$ for the filtration. Now we define a new quantum $\mathbb{Z}$-grading $\qgr$ on $HKh'(D)$ from the filtration.
\begin{df}\cite[Definition 2.5]{lobb2}
If $V$ is a filtered vector space $\cdots \subset \mathcal{F}^iV \subset \mathcal{F}^{i+1}V \subset \cdots$, for a non-zero element $x \in V$, we define the quantum grading qgr$(x) \in \mathbb{Z}$ when qgr$(x)$ satisfies $x \ne 0$ in $\mathcal{F}^{\text{qgr}(x)} V / \mathcal{F}^{\text{qgr}(x)-1} V$.
\end{df}
 In short, $\qgr$ is defined to be the smallest filtration level which contains the element. Moreover, we define a quantum $\mathbb{Z}/ 2n\mathbb{Z}$-grading $q_n$ on $CKh'(D)$. There is a canonical homomorphism $i$ from $\mathbb{Z}$ to $\mathbb{Z}/2n\mathbb{Z}$. Then $q_n$ is defined to be $i \circ q : CKh'(D) \longrightarrow \mathbb{Z}/2n\mathbb{Z}$.  $HKh'(D)$ has the quantum $\mathbb{Z}/2n\mathbb{Z}$-grading $q_n$, induced from $q_n$ on $CKh'(D)$. 
 
 Henceforth, we focus on the summand $${H}(D_\oo) \left\{ (1-n) w(D) \right\} =\mathbb{C} [ x_1, \cdots, x_r ] / (x_1^n -1, x_2^n -1, \cdots x_r^n -1) \left\{ (1-n) (w(D) + r) \right\},$$ which is supported in the homological grading $0$ subspace of $CKh'(D).$ 

 Let $a,b$ be two cycles in ${H} (D_{\oo}){(1-n)(w(D))} \subset CKh'_0 (D)$. Then it is proved that $$q_n ( a) - q_n (b) \equiv \qgr([a]) - \qgr([b]) \mod 2n$$ in \cite[Proposition 2.6]{lobb2}.

 We say $x \in {H}(D_\oo) \left\{ (1-n) w(D) \right\}$ is $\mathbb{Z}/2n\mathbb{Z}$-homogeneous when every term of $x$ has the same $q_n$ grading. We can represent the generators $g^D_0, \cdots, g^D_{n-1} \in  {H}(D_\oo) \left\{ (1-n) w(D) \right\}$ as a linear combination of $\mathbb{Z}/2n\mathbb{Z}$-homogeneous elements. 
\begin{df}
 We define $$h^D_p = \sum_{a_1 + \cdots a_r \equiv j \mod n} x_1^{a_1} \cdots x_r^{a_r} \in {H}(D_\oo) \left\{ (1-n) w(D) \right\} $$.
\end{df}
From \cite[Lemma 2.4]{lobb2}, $\displaystyle g^D_t= \sum_{j=0}^{n-1} {\xi^{-t(j+r)} h^D_j}$ when $\xi = \frac{2 \pi i}{n}$. Note that every coefficient is not zero. Conversely, $h^D_p$ also can be represented by a linear combination of $g^D_0, \cdots, g^D_{n-1}$. Therefore, $h^D_p$ is a cycle and $[h^D_p] \in HKh'(D)$. We can easily get that $q_n(h^D_p) = 2p + (1-n)(w(D) + r) \in \mathbb{Z}/ 2n \mathbb{Z}$. This gives 
\begin{equation}\label{qgr}\qgr ( [h^D_p]) \equiv 2p+ (1-n) (w(D) + r)  \mod 2n.\end{equation}

\begin{lem}\label{obs}
For a link diagram $D$, 
$$ \qgr ([g^D_i]) = \max \left\{ \qgr [h^D_i] : i=0, 1, \cdots n-1 \right\} $$ 
\end{lem}

\begin{proof}
 It is obvious since $g^D_i$ can be represented by a sum of a nonzero multiples of $h^D_0, \cdots, h^D_{n-1}$. 
\end{proof}
 When $D_1,D_2$ are two different oriented link diagrams of a link $L$, $\qgr( [g^{D_1}_i]) = \qgr( [g^{D_1}_i])$. This is straightforward since there is a homological grading and quantum filtration preserving isomorphism between $HKh'(D_1)$ and $HKh'(D_2)$ which maps $[g^{D_1}_i]$ to a nonzero multiple of $[g^{D_2}_i]$, from \cite[Proposition 5.9]{wu}. Now, we are ready to define $s_n$ for links. 
\begin{df}\label{defsn}
For a link $L$, we define a link invariant $$s_n(L)= \qgr([g^D_i]) - n + 1 \in \mathbb{Z}$$ for a link diagram $D$ of $L$ and $i \in \left\{0,1, \cdots, n-1\right\}$. This definition does not depend on the choice of $D$ and $i$. In particular, when $L$ is a knot, this is equivalent with Lobb's $s_n$ invariants. 
\end{df}

\section{Properties of $s_n$.} 

 In this section, we examine properties of new link invariants $s_n$. From the definition, we attain $s_n$ for splittable links.
\begin{prop}\label{prop2}
For a split link $L$, this can be uniquely represented by a union of $L_1, \cdots, L_m$ with non-splittable links $L_i$. Then $s_n (L) = s_n(L_1) + \cdots +  s_n (L_m) + (n-1)(m-1)$. In particular, for an $m$ component unlink $U_m$, $s_n (U_m) = (n-1)(m-1)$.
\end{prop}
\begin{proof}
 Let $D_i$ be a link diagram of $L_i$. Then $D=D_1 \sqcup \cdots \sqcup D_m$ is a link diagram for $L$. From Equation \ref{dis2}, $[g^{D}_i] =  [g^{D_1}_i] \otimes [g^{D_2}_i]\cdots \otimes [g^{D_m}_i]$. Therefore $$s_n (L) = \qgr ([g^{D}_i]) - (n-1) = \sum_{j=1}^m \qgr ([g^{D_j}_i]) - (n-1) = s_n (L_1) + \cdots s_n (L_m) + (n-1)(m-1).$$ For an unlink, $L_1, \cdots L_m$ are unknots. Therefore $s_n (L) = (n-1)(m-1)$.
\end{proof}

 We need the following two lemmas to prove the theorems stated in the introduction.

 \begin{lem}\label{lem2}
Suppose there is a link cobordism $\Sigma :L_1 \longrightarrow L_2$. We fix link diagrams $D_1, D_2$ for $L_1, L_2$. Then we construct the map $HKh'(\Sigma): HKh'(D_1) \longrightarrow HKh'(D_2)$ by a product of maps corresponding to elementary cobordisms explained in Section 3.3. If $HKh'(\Sigma)$ maps $[g^{D_1}_i]$ to a nonzero multiple of $[g^{D_2}_i]$, then $$s_n (L_1) - (n-1) \chi (\Sigma)\ge s_n (L_2).$$
\end{lem}
\begin{proof}
Since $HKh'(\Sigma)$ is a map of filtered degree $(1-n) \chi (\Sigma)$, $$\text{qgr} ([g^{D_1}_i]) - (n-1) \chi (\Sigma) \ge \text{qgr} ([g^{D_2}_i].$$ Therefore,
$$ s_n(L_1) + n -1 - (n-1)  \chi ({\Sigma})  \ge s_n (L_2) + n-1.$$
\end{proof}

\begin{lem}\label{prop1} Let $L$ be a link and $D$ be its diagram.There is an element $h^D\in {H} (D_{\oo}) \left\{ (1-n) w(D) \right\} \subset CKh' (D) $, which is a linear sum of nonzero multiples of $g^D_0, \cdots g^D_{n-1}$, satisfying $$\text{qgr} ([h^D]) \le s_n (L) - n + 1.$$\end{lem}
\begin{proof} 
 From Equation \ref{qgr}, it is straightforward that  $$ \max \left\{ \qgr[h^D_i] : i=0, 1, \cdots n-1 \right\} - \min \left\{ \qgr[h^D_i] : i=0, 1, \cdots n-1 \right\} \ge 2n-2.$$ Then, from Lemma \ref{obs}, $$\min \left\{ \qgr[h^D_i] : i=0, 1, \cdots n-1 \right\} \le \qgr([g^D_i]) - 2n + 2.$$
 Therefore, there exists an $i$ such that $h^D_i$ satisfies the conditions in this lemma.
\end{proof}

 Now we prove the theorems stated in the introduction.

\begin{proof}[Proof Theorem \ref{main}]
 First, let's prove the right side inequality. For an oriented surface $F$ in $4$ ball $B^4$ such that $\partial F = L \subset S^3$, we obtain a surface $F'$ by deleting a disk from $F$.  $F'$ can be considered as a cobordism $$\Sigma:L \longrightarrow U,$$ where $U$ is an unknot. Note that $\chi(\Sigma) = \chi(F)-1.$ Now we consider the ordering of elementary moves explained in Section 3.3. Let $k$ be the number of connected components of $F$. We start from $D_0$, a link diagram of $L$. 

\begin{enumerate}
\item all the $0$ handle moves.
\item Reidemeister moves.
\item A sequence of fusion $1$ handles ending in a $k$ component link diagram.
\item $g$ fission $1$-handles and $g$ fusion $1$-handles.
\item Reidemeister moves 
\item $(k-1)$ 2-handles. 
\end{enumerate}
 
 Compared with Theorem \ref{1.6}, the $3$rd step ends in a $k$ component link diagram instead of a knot diagram, since there are $k$ connected components. The first step maps $[g^{D_0}_i] \in HKh'(D_0)$ to $[g^{D_0}_i] \otimes 1 \cdots \otimes 1 \in HKh'(D_0) \otimes HKh'(U) \otimes \cdots \otimes HKh'(U)$. Note that $\displaystyle \sum_{j=0}^{n-1} (\prod_{\xi^k \in \Sigma_n \setminus \xi^j} \frac{1}{\xi^k - \xi^j}) g_j =1$. Let $D'$ be the link diagram we get after the first three steps.

 Let $S$ be a fusion $1$-handle map from $D \sqcup U$ to $D$ for a link diagram $D$. Then $HKh'(S)$ sends $$[g^{D}_i] \otimes 1 =\sum_{j=0}^{n-1} (\prod_{\xi^k \in \Sigma_n \setminus \xi^j} \frac{1}{\xi^k - \xi^j}) ([g^D_i] \otimes [g_j]) $$ to $\displaystyle \prod_{\xi^k \in \Sigma_n \setminus \xi^i} (\frac{1}{\xi^k - \xi^j}) [g^{D}_i]$, since $HKh'(S) ([g^D_i] \otimes [g_j]) = \begin{cases} [g^D_i] & \mbox{if $i=j$} \\ 0 &\mbox{otherwise} \end{cases}$. The 3rd step is a sequence of fusion $1$ handle maps whose descriptions are equivalent to $S$. Therefore, $[g^{D_0}_i]$ is mapped to a nonzero multiple of $[g^{D'}_i]$ after the first three steps. All elementary moves except the $0$-handle maps send $[g_i]$ to a nonzero multiple of $[g_i]$. From Lemma \ref{lem2}, $$s_n (L) - (n-1) \chi (\Sigma) \ge 0.$$

 Next, let's see the left side inequality. We obtain a surface $\Sigma'$ from $F$ by deleting one disk per each connected component. Then $\chi(\Sigma') = \chi(F) - k$. $\Sigma'$ is the cobordism from $k$ component unlink $U_k$ to $L$. Moreover, $HKh'(\Sigma')$ maps $[g^{U_k}_i]$ to $[g^D_i]$. Thus, $$(n-1)(k-1) - (n-1) \chi (\Sigma') \ge s_n (L).$$ 

\end{proof}

\begin{proof}[Proof of Theorem \ref{thm2}] We have the cobordism $\Sigma : L_0 \longrightarrow L_1$ which is a union of annuli. We fix link diagrams $D_0, D_1$ for $L_0, L_1$. Let $\Sigma = \Sigma_1 \sqcup \cdots \sqcup \Sigma_l$ when $\Sigma_i$ is a annulus, which has two boundaries such that one boundary is one component of $L_0$ and the other boundary is one component of $L_1$. $\Sigma$ can be represented by a product of elementary moves in the following order from Theorem \ref{1.6}. We don't need to have step $5,6$ in Theorem \ref{1.6}  since $\Sigma$ is a genus $0$ surface. 
\begin{enumerate}
\item all the $0$ handles. 
\item Reidemeister moves.
\item fusion $1$-handles, ending in an $l$ component link diagram $D'$.
\item Reidemeister moves.
\end{enumerate}
 The first step maps $[g^{D_0}_i]$ to $[g^{D_0}_i] \otimes 1 \otimes \cdots 1 \in HKh'(D_0) \otimes HKh'(U) \otimes \cdots HKh'(U)$. Then the Reidemeister moves map every element into itself in the homology level. The fusion maps in the $3$rd step send $[g^{D_0}_i] \otimes 1 \otimes \cdots 1$ to $[g^{D'}_i]$. Therefore, $HKh'(\Sigma)$ sends $[g^{D_0}_i]$ to $[g^{D_1}_i]$. 
 The link cobordism $\Sigma' : L_1 \longrightarrow L_0$ obtained from $\Sigma$ by flipping also satisfies the condition in Lemma \ref{lem2}. Therefore,
\begin{align*}
s_n (L_0) - (n-1) \chi ( \Sigma) \ge s_n (L_1), \\
s_n (L_1) - (n-1) \chi ( \Sigma') \ge s_n (L_0).
\end{align*}
$\chi(\Sigma) = \chi (\Sigma') = 0$ implies $$s_n (L_0) = s_n (L_1).$$

\end{proof}
  
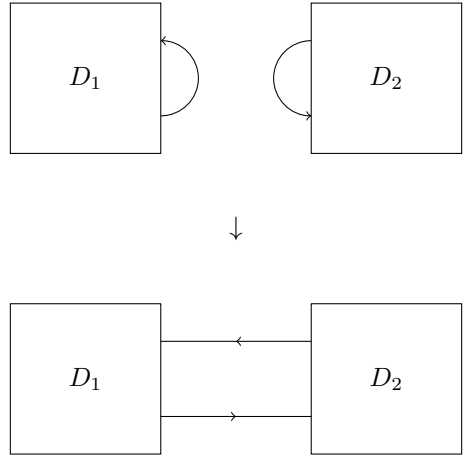
\begin{figure}
\begin{tikzpicture}
	\draw (1,1)--(3,1)--(3,3)--(1,3)--(1,1);
	\node at (2,2) {$D_2$};

	\draw (-1,1)--(-3,1)--(-3,3)--(-1,3)--(-1,1);
	\node at (-2,2) {$D_1$};

	\draw (1,-1)--(3,-1)--(3,-3)--(1,-3)--(1,-1);
	\node at (2,-2) {$D_2$};

	\draw (-1,-1)--(-3,-1)--(-3,-3)--(-1,-3)--(-1,-1);
	\node at (-2,-2) {$D_1$};
	\node at (0,0) {$\downarrow$};

 	\draw[->] (1,-1.5)--(0,-1.5);
	\draw (0,-1.5)--(-1,-1.5);
 	\draw (1,-2.5)--(0,-2.5);
	\draw[<-] (0,-2.5)--(-1,-2.5);
	
	\draw[->] (1,2.5) arc (90:270:0.5);
	\draw[->] (-1,1.5) arc (-90:90:0.5);

\end{tikzpicture}
	\caption{Connected sum of two links.}\label{conn} 
\end{figure}

\begin{proof}[Proof of Theorem \ref{thm3}]
First, (\ref{union}) is straightforward from Proposition \ref{prop2}. Next for (\ref{connected}), let $D_1, D_2$ be link diagrams for $L_1, L_2$. we have a 1-handle map from $D_1 \sqcup D_2$ to a connected sum of $D_1$ and $D_2$. The cobordism described in Figure \ref{conn} gives a map from $HKh'(L_1 \sqcup L_2)$ to $HKh'(L_1 \#_{i_1 = i_2} L_2).$ Since $1$-handle map has filtered degree $n-1$,
\begin{align*}
s_n (L_1) + s_n (L_2) + (n-1) + (n-1) &\ge s_n (L_1 \#_{i_1 = i_2} L_2). \\ 
s_n(L_1 \#_{i_1 = i_2} L_2) + n-1 &\ge s_n (L_1) + s_n (L_2) + (n-1).
\end{align*} 
  There exists a generator $[h^{D_2}] \in HKh' (D_2) $ satisfying $\text{qgr} ([h]) \le s_n(L_2) - n +1$ from Lemma \ref{prop1}. Then the above one handle cobordism from $L_1 \sqcup L_2$ to $L_1 \#_{i_1 = i_2} L_2$ maps $[g^{D_1}_i] \otimes [h^{D_2}]$ to a nonzero multiple of $[g^{D_1 \#_{i_1 = i_2} D_2}_i]$. Therefore, $$\text{qgr} ([g^{D_1}_i]) + \text{qgr} ([h^{D_2}]) + n-1  \ge \text{qgr} ([g^{D_1 \#_{i_1 = i_2} D_2}_i].$$ From the definition of $s_n$,
 \begin{equation*}
s_n(L_1) + n-1 + s_n (L_2) - n+1 +n-1 \ge s_n (L_1 \#_{i_1 = i_2} L_2) + n-1.
\end{equation*} hence $s_n(L_1) +s_n (L_2) \ge s_n(L_1 \#_{i_1 = i_2} L_2)$. Combining two inequalities, $$s_n(L_1) + s_n(L_2) = s_n (L_1 \#_{i_1 = i_2} L_2).$$ 

 For (\ref{mirror}), there is a link cobordism from $L \sqcup \bar{L}$ to the $l$ component unlink $U_n$. This cobordism consists of $l$ fusion 1-handles, which is filtered degree $(n-1)l$. This cobordism satisfies the condition in Lemma \ref{lem2}. Therefore, $$(n-1)(l-1) + (n-1)l \ge s_n (L) + s_n (\bar{L}) + (n-1).$$
 Then $(n-1)(2l-2) \ge s_n (L) + s_n(\bar{L})$.

 Furthermore, we pick $[h^D] \in HKh'(D)$ from Proposition \ref{prop1} when $D$ is a link diagram of $L$. Let $\bar{D}$ be the mirror of $D$. Then the cobordism from $L \sqcup \bar{L}$ to the $l$ component unlink maps $[h^D] \otimes [g^{\bar{D}}_i]$ to a nonzero multiple of $[g^{U_n}_i]$. Therefore, $\text{qgr} ([h^D]) + \text{qgr} ([g^{\bar{D}}_i]) + (n-1)l \ge \text{qgr}([g^{U_n}_i])$. From Lemma \ref{prop1}, $$s_n(L) + s_n(\bar{L}) \ge 0.$$

\end{proof}

\begin{proof}[Proof of Theorem \ref{positive}]
 Let's show that  $s_n (L) = (1-n)(c-r+1)$.  Since the image of the boundary map is empty in the homological grading 0, $\text{qgr} [g_i]$ is attained by the maximum quantum grading of $x_1^{n-1} \cdots x_r^{n-1} \in \mathbb{C} [ x_1, \cdots, x_r ] / (x_1^n -1, x_2^n -1, \cdots x_r^n -1) \left\{(1-n) (c + r)\right\}$. Therefore
\begin{center}
 $\text{qgr} [g_i]= 2r(n-1) + (1-n) (c+r) = (1-n) (c-r)$.
\end{center}

 We showed  $s_n (L) = (1-n)(c-r+1)$. 
 Then $c-r+1 \le 2g_4 (L) + l - 1 \le 2 g_3(L) + l-1 = c-r+1$ from Theorem \ref{main} and the fact $g_4 (L) \le g_3(L)$. Therefore every inequality becomes an equality. 
\end{proof}

\begin{proof}[Proof of Theorem \ref{change}]
 We can consider a link cobordism from $L_{+}$ to $L_{-}$ with a single double point. After resolving this double point, we can get a link cobordism $\Sigma$ from $L_{+}$ to $L_{-}$. We fix link diagrams $D_{+}$ and $D_{-}$ for $L_{+}$ and $L_{-}$. 

 First, suppose two strands of the crossing are in the same component. For other $l-1$ components which are not involved with the crossing, the cobordism is a topological cylinder. For the component which is involved with the crossing, the cobordism is a genus $1$ surface, which can be represented by a product of one fusion handle map and the other fission 1-handle map. Note that the fission map comes first. Therefore $HKh'(\Sigma)$ sends $[g^{D_{+}}_i]$ to a nonzero multiple of $[g^{D_{-}}_i]$  from the descriptions given in Section 3.1. In addition, this $\Sigma$ consists of $l-1$ annuli and genus 1 surface with 2 boundaries, hence $\chi(\Sigma) = -2$. Conversely, we can construct a link cobordism from $L_{-}$ to $L_{+}$ whose Euler characteristic is equal to $-2$. From Lemma \ref{lem2}, $$|s_n (L_{+}) - s_n (L_{-})| \le 2(n-1).$$

 Next, suppose two strands of the crossing are in the different component. Then, this link cobordism $\Sigma$ consists of $l-2$ annuli and genus $0$ surface with $4$ boundaries. This genus $0$ surface with $4$ boundaries is also a product of one fusion map and one fussion map. Note that the fusion map comes first in this case. $\chi (\Sigma) = -2$, thus $$|s_n (L_{+}) - s_n (L_{-})| \le 2(n-1).$$

\end{proof}

\begin{proof}[Proof of Theorem \ref{splitting}]
 First, from Theorem \ref{change},  $$| s_n (L) - s_n (L_1 \sqcup \cdots \sqcup L_l) | \le 2(n-1) sp(L).$$ Then, from Proposition \ref{prop2}, the statement is proved. 
\end{proof}

\begin{proof}[Proof of Corollary 3]
 First, we will show that $\displaystyle sp(L) \le \frac{l(l-1)}{2} pq$ by an induction on $l$. The base case is obvious since $sp(L)=0$ for one component link $L$. Then we assume that $\displaystyle sp(T((l-1)p, (l-1)q)) \le \frac{(l-1)(l-2)}{2} pq$ for $l>2$ as an induction hypothesis. We consider the link diagram $D$ of $L$ which is a closed braid with $lq$ strands. Then we mark one component. Among the crossings between the marked component and other components, we change the crossing if the marked component is under the crossing. We can easily check that $(l-1)pq$ crossings are switched. Then the marked component becomes split and the remaining part becomes $T((l-1)p, (l-1)q)$. From the induction hypothesis, $$sp(T(lp, lq)) \le (l-1) pq + \frac{(l-1)(l-2)}{2} pq = \frac{l(l-1)}{2} pq.$$

 Next, $L = L_1 \cup \cdots \cup L_l$ when $L_1, \cdots ,L_l$ are all torus knots $T(p,q)$. Therefore $$s_n (L) = (1-n) (lp \cdot lq - lp - lq +1)$$ and $$s_n(L_1) = \cdots = s_n(L_l) = (1-n) (pq - p - q +1)$$ from Theorem \ref{positive}. From Theorem \ref{splitting}, $$\frac{l(l-1)}{2} pq \le sp (L).$$ Therefore, the statement is proved.
\end{proof}

\begin{proof}[Proof of Theorem \ref{equal}]
 $s (L)$ is defined from Lee's homology which is constructed in \cite{lee}, while $s_2 (L)$ is defined from $HKh'_2 (L)$. We fix a link diagram $D$ of $L$. Then Lee's chain complex for $D$ is isomorphic $n=2$ Gornik's chain complex $HKh'_2(D)$, however there is a difference between the quantum grading conventions. The corresponding elements have the same quantum gradings with an opposite sign on the chain level. 

  We review the construction of $s(L)$ in \cite[Section 6]{beliakova}. $\mathbf{s}_{o}, \mathbf{s}_{\bar{o}}$ are elements corresponding to $[g^D_0]$ and $[g^D_1]$ respectively in Lee's homology. The Rasmussen invariant $s(L)$ of a link $L$ is defined to be
\begin{align*} s(L) &:= \frac{\deg(\mathbf{s}_{o} + \mathbf{s}_{\bar{o}}) + \deg(\mathbf{s}_{o}  - \mathbf{s}_{\bar{o}})}{2} \\
&= \max (\deg(\mathbf{s}_{o} + \mathbf{s}_{\bar{o}}), \deg(\mathbf{s}_{o}  - \mathbf{s}_{\bar{o}})) -1 \\
&= \min (\deg(\mathbf{s}_{o} + \mathbf{s}_{\bar{o}}), \deg(\mathbf{s}_{o}  - \mathbf{s}_{\bar{o}}))  + 1,\end{align*} where deg denotes the quantum filtration level. We remark that $\deg(\mathbf{s}_{o} + \mathbf{s}_{\bar{o}}) = - \qgr ([g_0]+[g_1])$, $\deg(\mathbf{s}_{o}  - \mathbf{s}_{\bar{o}}) = - \qgr ( [g_0] - [g_1]).$

  We examine the $s_2 (L)$. Henceforth, let's omit the superscript $D$ in $g^D_i$ and $h^D_i$ for $i=0,1$. From Section 4, $g_0 = h_0 + h_1, g_1 = \pm (h_0 - h_1).$ Then  $$s_2 (L) = \qgr ([g_0]) - 1 = \max \left\{ \qgr([h_0] , \qgr[h_1] \right\}- 1  = \max \left\{ \qgr([g_0] + [g_1]), \qgr([g_0] - [g_1])\right\}- 1.$$ Therefore $s_2 (L) = - s (L).$
\end{proof}

\printbibliography
\end{document}